\documentclass[11pt,twoside,a4paper]{amsart}
\usepackage{amsfonts}
\usepackage{amssymb}
\usepackage{amsmath}

\title[Diederich--Fornaess index of weakly pseudoconvex domains]
{A global estimate for the Diederich--Fornaess index of weakly pseudoconvex domains}

\author{Masanori Adachi}
\address{Graduate~School~of~Mathematics, Nagoya~University, Furo-cho Chikusa-ku Nagoya 464-8602, Japan}
\email{m08002z@math.nagoya-u.ac.jp}

\author{Judith Brinkschulte}
\address{Universit\"at Leipzig, Mathematisches Institut, PF 100920, D-04009 Leipzig, Germany}
\email{brinkschulte@math.uni-leipzig.de}

\date{March 26, 2015.}

\thanks{This is the author's final version of an article accepted for publication in Nagoya Mathematical Journal. The editorial board of Nagoya Mathematical Journal is the copyright holder of the published version of this article.}

\newcommand{\pa}{\partial}
\newcommand{\opa}{\overline\pa}
\newcommand{\ol}{\overline }

\newcommand\C{\mathbb{C}}  
\newcommand\R{\mathbb{R}}

\newcommand\PP{\mathbb{P}}

\makeatletter
\newsavebox{\@brx}
\newcommand{\llangle}[1][]{\savebox{\@brx}{\(\m@th{#1\langle}\)}%
  \mathopen{\copy\@brx\kern-0.5\wd\@brx\usebox{\@brx}}}
\newcommand{\rrangle}[1][]{\savebox{\@brx}{\(\m@th{#1\rangle}\)}%
  \mathclose{\copy\@brx\kern-0.5\wd\@brx\usebox{\@brx}}}
\makeatother

\newtheorem*{MainTheorem}{Main Theorem}
\newtheorem*{Claim}{Claim}

\newtheorem{Theorem}{Theorem}[section]

\newtheorem{Lemma}[Theorem]{Lemma}
\newtheorem{Corollary}[Theorem]{Corollary}
\numberwithin{equation}{section}

\begin{document}
\maketitle

%%%%%%%%%%%%%%%%%%%%%%%%%%%%%%%%%%%%%%%%%%%%%%%%%%%%%%%%%%%%%%%%%%%%%%%%%%%%%%%%%%%
\begin{abstract}
A uniform upper bound for the Diederich--Fornaess index is given for weakly pseudoconvex domains 
whose Levi-form of the boundary vanishes in $\ell$-directions everywhere.
\end{abstract}

%%%%%%%%%%%%%%%%%%%%%%%%%%%%%%%%%%%%%%%%%%%%%%%%%%%%%%%%%%%%%%%%%%%%%%%%%%%%%%%%%%%
\section{Introduction}

The aim of this paper is to reveal a relation between the Diederich--Fornaess index
of weakly pseudoconvex domains and the rank of the Levi-form of their boundaries.

Let us first recall the definition of the Diederich--Fornaess index. 
Consider a complex manifold $X$ and a relatively compact domain $\Omega \Subset X$ with $\mathcal{C}^2$-smooth boundary.
A defining function of $\Omega$ is 
a $\mathcal{C}^2$-smooth function $\rho: \ol\Omega \rightarrow\R$ satisfying $\Omega = \{\rho < 0 \}$ and 
whose gradient does not vanish on $\pa\Omega$. In order to avoid too many minus signs in this paper, 
 we will associate to a fixed defining function $\rho$ the nonnegative function $\hat{\delta} =\hat{\delta}_\rho = -\rho$, which can be thought of as a boundary distance function of $\Omega$ with respect to a certain hermitian metric on $X$ (depending on $\rho$).

The {\em Diederich--Fornaess exponent} $\eta_{\hat{\delta}}$ of a defining function $-{\hat{\delta}}$ is the supremum of $\eta \in (0, 1)$ such that 
$-{\hat{\delta}}^{\eta}$ is a bounded, strictly plurisubharmonic exhaustion function of $\Omega$. 
If there is no such $\eta$, we let $\eta_{\hat{\delta}} := 0$.
The {\em Diederich--Fornaess index} $\eta(\Omega)$ of $\Omega$ is the supremum of the Diederich--Fornaess exponents of defining functions of $\Omega$.

The Diederich--Fornaess index is a numerical index on 
the strength of a certain pseudoconvexity, more precisely that of hyperconvexity.
If $\pa \Omega$ is strictly pseudoconvex, we know that $\pa \Omega$ admits a strictly plurisubharmonic defining function, 
hence, $\eta(\Omega) = 1$. 
In order for $\Omega$ to have positive $\eta(\Omega)$, $\Omega$ must be Stein, and we need more in fact:
A theorem of Ohsawa--Sibony (\cite{OSi}; see also \cite{HSh}) tells us that $\eta_{\hat{\delta}} > 0$ 
if and only if $i\pa\opa(- \log {\hat{\delta}}) \geq \omega_0$ in $\Omega$ for some hermitian metric $\omega_0$ of $X$. 
The domains $\Omega$ with positive $\eta(\Omega)$ should carry such a special exhaustion 
as if they are proper pseudoconvex domains in $X=\mathbb{CP}^n$, where Takeuchi's theorem guarantees this kind of exhaustions.
Many techniques using such exhaustions have been developed for solving the $\opa$-equation on weakly pseudoconvex domains, 
see for example \cite{BCh}, \cite{Br}, \cite{CShW}, \cite{HSh}, \cite{CSh}.

Let us give several examples to illustrate the situation we are considering. 
In the celebrated series of works \cite{DiFo1}, \cite{DiFo2} of Diederich and Fornaess,
they showed that if $X$ is Stein, $\eta(\Omega) > 0$ for any domain $\Omega \Subset X$ with $\mathcal{C}^2$-smooth pseudoconvex boundary. 
Note that in this situation $\pa\Omega$ must have a strictly pseudoconvex point, 
for we can find a level set of a strictly plurisubharmonic exhaustion of $X$ touching $\pa\Omega$ at some points and bounding $\Omega$. 
They also showed that for any $\varepsilon > 0$, there is $\Omega \Subset X=\C^2$ with $0 < \eta(\Omega) < \varepsilon$
by using the worm domains, where a Levi-flat portion sits on $\pa\Omega$.   
The first author proved in \cite{A} that certain holomorphic disc bundles $\Omega$ over compact Riemann surfaces in their associated 
flat ruled surfaces $X$ satisfy $\eta(\Omega) > 0$ even though $\pa\Omega$ is totally Levi-flat. 

% for any boundary point, we can find a neighborhood $U$ where $\eta(U \cap \Omega)$ is arbitrarily close to 1.

A natural question therefore is to ask to what extent the Diederich--Fornaess exponent gets smaller 
when $\pa\Omega$ is nearly Levi-flat everywhere. 
Our answer is the following 

%%%%%
\begin{MainTheorem}
Let $X$ be a complex manifold of dimension $n \geq 2$ and $\Omega \Subset X$ a relatively compact domain with $\mathcal{C}^3$-smooth boundary. 
Assume that the Levi-form of the boundary $\pa\Omega$ has at least $\ell$ zero eigenvalues everywhere on $\pa\Omega$ where $0\leq\ell\leq n-1$. 
Then $\eta(\Omega) \leq \frac{n-\ell}{n}$. 
\end{MainTheorem}

In particular, we obtain 
\begin{Corollary}
If  $\eta(\Omega) > \frac{1}{n}$, then $\pa\Omega$ is not Levi-flat.
\end{Corollary}

\begin{Corollary}
If $\eta(\Omega) > \frac{n-1}{n}$, then $\pa\Omega$ has a strictly pseudoconvex point.\\
\end{Corollary}

%%%%%
Let us explain the idea of our proof of Main Theorem. When $X$ is Stein, we found a strictly pseudoconvex point on $\pa\Omega$ 
by approximating $\pa\Omega$ at a point by strictly pseudoconvex real hypersurfaces from outside.
Since no such approximation exists in general, we use the following method inside: 
We assume by contradiction that $\eta(\Omega) > \frac{n-\ell}{n}$. 
Then we show in Theorem \ref{vanishing}, using weighted $L^2$-estimates, that any smooth, top-degree form with compact support in $\Omega$ 
is $\opa$-exact in the sense of currents on $X$. 
This is impossible essentially because the top-degree cohomology with compact support does not vanish.

%%%%%
For the proof of Theorem \ref{vanishing}, we use an estimate of Donnelly--Fefferman type (see \cite{DoFe}) to pass  from an $L^2$ vanishing result in $L^2_{n,n}(\Omega,{\hat{\delta}}^{\eta})$ to an 
$L^2$ vanishing result in $L^2_{n,n}(\Omega,{\hat{\delta}}^{-\eta})$. We also modify this argument by using a special  K\"ahler metric $\omega := i\pa\opa(-{\hat{\delta}}^\eta)$ in $\Omega$ for some $\eta \in (0, \eta_{\hat{\delta}})$. This metric respects the degeneracy of the Levi-form of $\pa\Omega$ in a certain manner and permits to prove that the trivial extension of this solution is in fact a solution on all of $X$.

\subsection*{Acknowledgements}
After this work was accomplished, the authors were kindly informed by Siqi Fu and Mei-Chi Shaw
that they had already reached the same result for a weaker assumption in \cite{FuSh}  
with a different technique. We are grateful to them and Takeo Ohsawa for communicating this information.

We also thank a referee for his/her careful reading and comments to improve the presentation of this paper.

%%%%%%%%%%%%%%%%%%%%%%%%%%%%%%%%%%%%%%%%%%%%%%%%%%%%%%%%%%%%%%%%%%%%%%%%%%%%%%%%%%%
\section{Preliminaries on $L^2$-estimate}

In this section we introduce some notations that are used in the sequel. Also, for the convenience of the reader, we recall some of the basic facts concerning a priori estimates and solvability results for the $\opa$ operator.\\

%%%%%
Let $X$ be a complex manifold equipped with a hermitian metric $\omega_0$ 
and $\Omega \subset X$ a domain with $\mathcal{C}^2$-smooth boundary. 
We let $-{\hat{\delta}}: \ol\Omega \rightarrow\R$ be a defining function.

We denote by $L^2_{p,q}(\Omega,{\hat{\delta}}^s)$ the Hilbert space of $(p,q)$-forms $u$  which satisfy
$$\Vert u\Vert_{{\hat{\delta}}^s}^2 := \int_\Omega \vert u\vert^2_{\omega_0} {\hat{\delta}}^s dV_{\omega_0} < + \infty.$$
Here $dV_{\omega_{0}}$ is the canonical volume element associated with the metric $\omega_{0}$, and $\vert\cdot\vert_{\omega_{0}}$ is the norm of $(p,q)$-forms induced by $\omega_{0}$. For $s=0$ the $L^2$-spaces just defined coincide with the usual $L^2$-spaces on $\Omega$; in this case, we will omit the index ${\hat{\delta}}^0$.\\

%%%%%
In our proofs it is sometimes necessary to replace the base metric $\omega_{0}$ by a different metric $\omega$.
The corresponding Hilbert spaces resp. norms will then be denoted by 
$L^2_{p,q}(\Omega,{\hat{\delta}}^s, \omega)$ resp. $\Vert \cdot\Vert_{{\hat{\delta}}^s,\omega}$.\\

%%%%%
For later use, we recall the well known Bochner--Kodaira--Nakano inequality for K\"ahler metrics for the special case of the trivial line bundle $\C$ on $\Omega$ equipped with a weight function $\varphi\in \mathcal{C}^2(\Omega)$, which is the key point when establishing $L^2$ existence theorems for the $\opa$ operator (see \cite{D}):

Let $\omega$ be a  K\"ahler metric on $\Omega$. Then for every $u\in \mathcal{D}^{p,q}(\Omega)$ we have
\begin{equation}
\label{BKN}
\Vert \opa u\Vert^2_{e^{-\varphi}} + \Vert \opa^\ast_{e^{-\varphi}} u\Vert^2_{e^{-\varphi}} 
\geq \llangle \lbrack i\pa\opa\varphi,\Lambda\rbrack u,u \rrangle_{e^{-\varphi}}. 
\end{equation}
Here $\Lambda$ is the adjoint of multiplication by $\omega$.\\

%%%%%
A standard computation for the curvature term yields that
\begin{equation}
\label{negativecurvature}
\langle \lbrack i\pa\opa\varphi,\Lambda\rbrack u,u \rangle \geq (\lambda_1 + \ldots + \lambda_q-\sum_{j=1}^n\lambda_j) \vert u\vert^2
\end{equation}
for any form $u\in \Lambda^{0,q}T^\ast \Omega$.
Here $\lambda_1\leq \ldots\leq \lambda_n$ are the eigenvalues of $i\pa\opa\varphi$ with respect to $\omega$.\\

%%%%%%%%%%%%%%%%%%%%%%%%%%%%%%%%%%%%%%%%%%%%%%%%%%%%%%%%%%%%%%%%%%%%%%%%%%%%%%%%%%%
\section{A special metric}

%%%%%
When $\Omega$ has a defining function $-{\hat{\delta}}$ with positive Diederich--Fornaess exponent $\eta_{\hat{\delta}}$, 
taking $0 < \eta < \eta_{\hat{\delta}}$,  
we will equip the domain $\Omega$ with another K\"ahler metric $\omega := i\pa\opa(-{\hat{\delta}}^\eta)$ different from $\omega_0$. 

Let us study the behavior of the metric $\omega$ near $\pa \Omega$ for later use. 
\begin{Lemma}
Suppose that $\pa\Omega$ is $\mathcal{C}^3$-smooth and 
the Levi-form of $\pa\Omega$ has at least $\ell$ zero eigenvalues everywhere. 
Then, we have 
\begin{equation}
\label{volumedev}
dV_\omega \lesssim {\hat{\delta}}^{n\eta -2 - (n - \ell - 1)} dV_{\omega_0}
\end{equation}
near $\pa \Omega$.
\end{Lemma}

\begin{proof}
First fix a finite covering of $\pa\Omega$ by holomorphic charts $\{ (U; z_U) \}$ 
equipped with the Euclidean metrics $\omega_{U}$ associated with their coordinates $z_U$. 
We can fix the covering so that 
\begin{itemize}
\item $| d{\hat{\delta}} |_{\omega_U} > 1 $ on each chart $U$; 
\item $\omega_U$ are uniformly comparable to $\omega_0$; 
\item a $\mathcal{C}^k$-norm for functions defined on a neighborhood of $\ol{\Omega}$, say $\| \cdot \|_{\mathcal{C}^k(\ol{\Omega})}$,
bounds the $\mathcal{C}^k$-norm associated with the coordinate $z_U$ from above for functions 
compactly supported in $U$. 
\end{itemize}

Let $p\in\pa\Omega$ and take one of the holomorphic charts that contains $p$, say $(U; z_U=(z_1, z_2, \cdots, z_n))$. 
For small $\varepsilon > 0$, consider a non-tangential cone 
$\Gamma_{p,\varepsilon} := \{ z \in U \cap \Omega \mid |z - p| < 2 {\hat{\delta}}(z), |z - p| < \varepsilon \}$ 
with vertex at $p$. Note that $\Gamma_{p,\varepsilon}$ is non-empty 
as $\ol{\Gamma}_{p,\varepsilon}$ contains a segment starting from $p$ normal to $\ker d{\hat{\delta}}_p$.
It suffices to find a positive constant $C$ independent of the choice of $p$ so that 
\[
D_U := \frac{dV_{\omega}}{dV_{\omega_U}} \leq C {\hat{\delta}}^{n\eta - 2 - (n - \ell - 1)}
\]
holds on $\Gamma_{p,\varepsilon}$ for some $\varepsilon = \varepsilon(p) > 0$.
That is because $\bigcup_{p \in \pa\Omega} \Gamma_{p,\varepsilon(p)} = W \cap \Omega$ for some neighborhood $W$ of $\pa\Omega$
and $\omega_0$ is comparable to every $\omega_U$ with a uniform constant; 
we can prove the desired inequality on $W \cap \Omega$. 
\\

To compute $dV_{\omega}/dV_{\omega_U}$, we will select an orthonormal frame of $T^{1,0}U$. 
By a unitary transformation, we can suppose $\ker d{\hat{\delta}}_p = \C^{n-1} \times \R$ and $\C^{\ell} \times \{0'\}$ 
is contained in the kernel of the Levi form of $\pa\Omega$ at $p$.
Define a $\mathcal{C}^2$-smooth frame $\mathcal{Y} = (Y_1, Y_2, \cdots, Y_n)$ of $T^{1,0}U$ by  
\[
Y_j := \frac{\pa}{\pa z_j} -  \frac{\frac{\pa{\hat{\delta}}}{\pa z_j}}{\frac{\pa{\hat{\delta}}}{\pa z_n}}\frac{\pa}{\pa z_n}\, (j = 1, 2, \cdots, n-1),
\quad Y_n := \frac{\pa}{\pa z_n}.
\]
Note that $\{Y_1, Y_2, \cdots, Y_{n-1} \}$ spans $\ker \pa{\hat{\delta}}$ on $U$. 
We apply the Gram--Schmidt procedure to $\mathcal{Y}$ and obtain an orthonormal frame $\mathcal{X} = ( X_1, X_2, \cdots, X_{n} )$ 
with respect to $\omega_U$. 
Denote by $A(z) = (a_{jk}(z))$ the change-of-base matrices at each point: $X_k = \sum_{j=1}^n Y_j a_{jk}$ on $U$.

We would like to estimate each $\lambda_{j\ol{k}} := \omega(X_j, \ol{X_k})$ on $\Gamma_{p,\varepsilon}$.
To achieve it, we combine two estimates: one is about $\mu_{j\ol{k}} := \omega(Y_j, \ol{Y_k})$ and 
the other is about the change-of-base matrices $A(z)$. 

First consider the behavior of $\mu_{j\ol{k}}$ on $\Gamma_{p,\varepsilon}$.
The equality
\begin{equation}  \label{metriccomputation}
\omega = i\eta{\hat{\delta}}^{\eta} \left\lbrace \frac{\pa\opa(-{\hat{\delta}})}{{\hat{\delta}}} + (1-\eta) \frac{\pa{\hat{\delta}}\wedge\opa{\hat{\delta}}}{{\hat{\delta}}^2} \right\rbrace
\end{equation}
yields that if $j = k = n$, 
\[
\lim_{z \to p, z \in U\cap\Omega} \frac{\mu_{n\ol{n}}(z)}{{\hat{\delta}}(z)^{\eta-2}} = \eta(1-\eta)|\pa{\hat{\delta}}(Y_n(p))|^2 \leq  \| {\hat{\delta}} \|_{\mathcal{C}^1(\ol{\Omega})}^2;
\]
otherwise,
\[
\lim_{z \to p, z \in U\cap\Omega} \frac{|\mu_{j\ol{k}}(z)|}{{\hat{\delta}}(z)^{\eta-1}} = \eta|\pa\opa(-{\hat{\delta}})(Y_j(p), \ol{Y_k(p)})| \leq  \| {\hat{\delta}} \|_{\mathcal{C}^2(\ol{\Omega})}.
\]
We can say more for directions in which the Levi-form vanishes. 
If $1 \leq j \leq \ell$, $1 \leq k \leq n-1$ or $1 \leq j \leq n-1$, $1 \leq k \leq \ell$, 
\begin{align*}
\limsup_{z \to p, z \in \Gamma_{p,\varepsilon}} \frac{|\mu_{j\ol{k}}(z)|}{{\hat{\delta}}(z)^{\eta} } 
& = \limsup_{z \to p, z \in \Gamma_{p,\varepsilon}} \eta\left|\frac{\pa\opa(-{\hat{\delta}})(Y_j(z), \ol{Y_k(z)})}{{\hat{\delta}}(z)}\right| \\
& = \limsup_{z \to p, z \in \Gamma_{p,\varepsilon}} \eta \frac{|z-p|} {{\hat{\delta}}(z)} \left|\frac{\pa\opa(-{\hat{\delta}})(Y_j(z), \ol{Y_k(z)}) - 0}{|z-p|} \right| \\
& \leq 2 | d \left(\pa\opa(-{\hat{\delta}})(Y_j, \ol{Y_k})\right)(p) |_{\omega_U} \\
& \leq 2 (\| {\hat{\delta}} \|_{\mathcal{C}^3(\ol{\Omega})} + 2\| {\hat{\delta}} \|^2_{\mathcal{C}^2(\ol{\Omega})}). 
\end{align*}

Next we proceed to estimate the change-of-base matrices $A(z)$. 
We identify an $n$-tuple of $(1,0)$-vectors with an $n \times n$ matrix by using our coordinate $z_U$. 
Then, we have $\mathcal{X}(p) = \mathcal{Y}(p) = I_n$ and $A(z) =\mathcal{Y}^{-1}(z) \cdot \mathcal{X}(z)$ 
where $I_n$ denotes the identity matrix. 
As a matrix-valued 1-form, we have 
\[
dA(p) = \mathcal{Y}^{-1}(p) \cdot d\mathcal{X}(p) + d\mathcal{Y}^{-1}(p) \cdot X(p) =  d\mathcal{X}(p) + d\mathcal{Y}^{-1}(p).
\]
Since $ I_n = \mathcal{Y}^{-1}(z) \cdot \mathcal{Y}(z) $, we also have
\[
0 = d(\mathcal{Y}^{-1} \cdot \mathcal{Y})(p) = d\mathcal{Y}^{-1}(p) + d\mathcal{Y}(p).
\]
Now let $\mathrm{GS} : GL(n, \C) \to U(n)$ be the map determined by the Gram--Schmidt procedure. 
Its differential at $I_n$ defines $d\mathrm{GS}_{I_n}: \mathfrak{gl}(n, \C) \to \mathfrak{u}(n)$. 
We linearly extend this map on matrix-valued, i.e., $\mathfrak{gl}(n, \C)$-valued 1-forms and 
also write $d\mathrm{GS}_{I_n}$ for the extended linear map by abuse of notation.
Then, $d \mathrm{GS}_{I_n} \left(d\mathcal{Y}(p) \right) = d\mathcal{X}(p)$ follows from $GS(\mathcal{Y}(z)) = \mathcal{X}(z)$.
Combining these equalities, we therefore have
\[
dA(p) = d \mathrm{GS}_{I_n} \left(d\mathcal{Y}(p) \right) - d \mathcal{Y}(p).
\]

We use the norm $| A | = \max_{j,k} |a_{jk}|$ for matrices and consider the induced norm for linear maps between spaces of matrices. 
Since a straightforward computation yields $|d \mathcal{Y}(p) |_{\omega_U} \leq \|{\hat{\delta}}\|_{\mathcal{C}^2(\ol{\Omega})}$, we have 
\begin{align*}
\limsup_{z \to p, z \in \Gamma_{p,\varepsilon}} \frac{ |A(z) - I_n|}{{\hat{\delta}}(z)} 
& = \limsup_{z \to p,  z \in \Gamma_{p,\varepsilon}} \frac{|z-p|}{{\hat{\delta}}(z)} \frac{ |A(z) - I_n|}{|z - p|}  \\
& \leq   2 | dA(p) |_{\omega_U} \\
& \leq 2 (| d \mathrm{GS}_{I_n} | + 1 ) | d\mathcal{Y}(p) |_{\omega_U} \\
& \leq 2 (| d \mathrm{GS}_{I_n} | + 1 ) \| {\hat{\delta}} \|_{\mathcal{C}^2(\ol{\Omega})}.
\end{align*}
Note that $| d \mathrm{GS}_{I_n} |$ is independent of $p$ and depends only on $n$.

By combining the estimates on  $\mu_{j\ol{k}}$ and $A(z)$ above, we can find a positive constant $C$ depending only on 
$n = \dim X$ and $\| {\hat{\delta}} \|_{\mathcal{C}^3(\ol{\Omega})}$
so that 
\begin{align}
\label{metricdev}
|\lambda_{j\ol{k}}(z)| &= \left| \sum_{l,m} \mu_{l\ol{m}}(z) a_{jl}(z) \ol{a_{km}(z)} \right| \\
&\leq
\begin{cases}
  C {\hat{\delta}}^{\eta-2} & (\text{for $j = k = n$}) \\
  C {\hat{\delta}}^{\eta}   & (\text{for $1 \leq j \leq \ell,\, 1\leq k \leq n-1$}) \\
  C {\hat{\delta}}^{\eta}   & (\text{for $1 \leq j \leq n-1, 1 \leq k \leq \ell$}) \\
  C {\hat{\delta}}^{\eta-1} & (\text{otherwise}) \\
\end{cases} \nonumber
\end{align}
holds on $\Gamma_{p,\varepsilon}$ for $0 < \varepsilon \ll 1$. 
It follows that
\begin{align*}
D_U
&= \det \left(\lambda_{j\ol{k}}\right)_{j,k=1}^n \\
&\leq n! C^n {\hat{\delta}}^{\ell \eta + (n-\ell-1)(\eta - 1) + (\eta - 2)}\\
&= n! C^n {\hat{\delta}}^{n\eta - 2 - (n - \ell - 1)} 
\end{align*}
on $\Gamma_{p,\varepsilon}$, which completes the proof.  
\end{proof}

%%%%%
\begin{Lemma}
\label{formdev}
Suppose that $\pa \Omega$ is $\mathcal{C}^3$-smooth and the Levi-form of $\pa\Omega$ has at least $\ell$ zero eigenvalues everywhere. 
Then, for any $(n,n-1)$-form $u$ on $\Omega$, 
\[
|u|^2_{\omega_0} dV_{\omega_0} \lesssim |u|^2_\omega {\hat{\delta}}^{(n-1)\eta-2-(n- \ell -1)} dV_\omega 
\]
near $\pa \Omega$ with positive constant independent of $u$.
\end{Lemma}

\begin{proof}
It suffices to prove the inequality on $\Gamma_{p,\varepsilon}$ with $\omega_U$ instead of $\omega_0$ 
where we work in the same local situation as in the proof of Lemma \ref{volumedev}. 
Consider the induced frame of $\wedge^n T^{1,0}U \otimes \wedge^{n-1} T^{0,1}U$ from $\{X_1, X_2, \cdots, X_n\}$ over $U$.
It follows from (\ref{metricdev}) that 
\begin{align*}
&|X_1 \wedge X_2\wedge \cdots\wedge X_n \otimes \ol{X}_1\wedge \ol{X}_2\wedge \cdots\wedge \widehat{\ol{X}}_{k}\wedge \cdots\wedge \ol{X}_n|^2_\omega \\
&= D_U |\ol{X}_1\wedge \ol{X}_2\wedge \cdots\wedge \widehat{\ol{X}}_{k}\wedge \cdots\wedge \ol{X}_n|^2_\omega \\
&\leq  D_U (n-1)! C^{n-1}  \begin{cases}
{\hat{\delta}}^{(\ell-1)\eta + (n-\ell-1)(\eta - 1) + (\eta - 2)} &\text{(for $1 \leq k \leq \ell$)}\\
{\hat{\delta}}^{\ell\eta + (n-\ell-2)(\eta - 1) + (\eta -2)} &\text{(for $\ell + 1 \leq k \leq n-1$)}\\
{\hat{\delta}}^{\ell\eta + (n-\ell-1)(\eta - 1)} &\text{(for $k = n$)}\\
\end{cases}\\
&\leq   D_U (n-1)! C^{n-1}{\hat{\delta}}^{(n-1)\eta - 2 - (n-\ell-1)}.
\end{align*}
Hence, we can estimate $|u|^2_{\omega}$ as 
\begin{align*}
|u|^2_{\omega} 
&\geq \max_{1 \leq k \leq n} \frac
{|u(X_1, X_2, \cdots, X_n, \ol{X}_1, \ol{X}_2, \cdots, \widehat{\ol{X}}_{k}, \cdots, \ol{X}_n)|^2}
{|X_1 \wedge X_2\wedge \cdots\wedge X_n \otimes \ol{X}_1\wedge \ol{X}_2\wedge \cdots\wedge \widehat{\ol{X}}_{k}\wedge \cdots\wedge \ol{X}_n|^2_\omega}\\
&\geq\frac
{\max_{1 \leq k \leq n} |u(X_1, X_2, \cdots, X_n, \ol{X}_1, \ol{X}_2, \cdots, \widehat{\ol{X}}_{k}, \cdots, \ol{X}_n)|^2}
{(n-1)! C^{n-1} D_U {\hat{\delta}}^{(n-1)\eta - 2 - (n-\ell-1)}} \\
&\geq C' \frac
{|u|^2_{\omega_U}}
{D_U {\hat{\delta}}^{(n-1)\eta - 2 - (n-\ell-1)}}
\end{align*}
with constant $C'>0$ independent of $u$. We therefore have the desired inequality 
\begin{align*}
|u|^2_\omega dV_\omega 
&\geq C' \frac{1}{D_U}|u|^2_{\omega_U} {\hat{\delta}}^{-(n-1)\eta + 2 + (n-\ell-1)} D_U dV_{\omega_U}\\
&= C' |u|^2_{\omega_U} {\hat{\delta}}^{-(n-1)\eta + 2 + (n-\ell-1)} dV_{\omega_U}.
\end{align*}
\end{proof}

%%%%%%%%%%%%%%%%%%%%%%%%%%%%%%%%%%%%%%%%%%%%%%%%%%%%%%%%%%%%%%%%%%%%%%%%%%%%%%%%%%%
\section{The $\opa$ equation in top degree}

In this section, we will study a version of an $L^2$ $\opa$-Cauchy problem in top degree 
on a smoothly bounded domain with weakly pseudoconvex boundary, which, by duality, implies 
a restriction on the rank of the Levi-form of $\pa \Omega$.

\begin{Theorem}   \label{vanishing}
Let $X$ be a complex manifold of dimension $n \geq 2$ and $\Omega \Subset X$ a relatively compact domain with $\mathcal{C}^3$-smooth boundary. 
Suppose that the Levi-form of $\pa \Omega$ has at least $\ell$ zero eigenvalues everywhere on $\pa\Omega$ for some $0\leq \ell\leq n-1$.
If $\eta(\Omega) > \frac{n-\ell}{n}$, 
then for any $f\in L^2_{n,n}(X)$  which is compactly supported in $\Omega$, 
there exists a current $T\in\mathcal{D}^\prime_{0,1}(X)$ with supported in $\ol{\Omega}$ such that $\opa T = f$ in the distribution sense on $X$.
\end{Theorem}

%%%%%
Theorem \ref{vanishing} is based on the following estimate of Donnelly--Fefferman type. 

\begin{Theorem}  \label{donnelly-fefferman} 
Let $X$ be a complex manifold of dimension $n \geq 2$ and $\Omega \Subset X$ a relatively compact domain with $\mathcal{C}^2$-smooth boundary. 
Let $-{\hat{\delta}}$ be a defining function of $\Omega$ with Diederich--Fornaess exponent $\eta_{\hat{\delta}} > 0$. 
For an arbitrary but fixed $\eta \in (0, \eta_{\hat{\delta}})$ we define
 $\omega := i\pa\opa(-{\hat{\delta}}^\eta)$.
Then, for any $f\in L^2_{n,n}(\Omega,{\hat{\delta}}^{-\eta},\omega)$, 
there exists $u\in  L^2_{n,n-1}(\Omega,{\hat{\delta}}^{-\eta},\omega)$ satisfying $\opa u = f$ in the distribution sense in $\Omega$.
\end{Theorem}

\begin{proof}
Let us first see that the conclusion follows in a standard manner from the following a priori estimate:

\begin{Claim}
There exists a constant $C>0$ such that 
\begin{equation}  \label{apriori}
\Vert v\Vert^2_{{\hat{\delta}}^{-\eta},\omega}\leq C \Vert \opa^\ast v\Vert^2_{{\hat{\delta}}^{-\eta},\omega}
\end{equation}
for any $v\in \mathcal{D}^{n,n}(\Omega)$. Here $\opa^\ast = \opa^\ast_{{\hat{\delta}}^{-\eta},\omega}$ is the adjoint of $\opa$ with respect to the scalar product induced by $\Vert\cdot\Vert_{{\hat{\delta}}^{-\eta},\omega}$.
\end{Claim}

Note that in the top degree, we can work with non-complete metrics, since there is no compatibility condition. 
Indeed, let us take $f\in L^2_{n,n}(\Omega,{\hat{\delta}}^{-\eta},\omega)$ and 
define a linear functional $\phi$ on $\opa^*(\mathcal{D}^{n,n}(\Omega)) \subset L^2_{n,n-1}(\Omega, {\hat{\delta}}^{-\eta},\omega)$ 
by $\phi(\opa^*v) = \llangle v, f \rrangle_{{\hat{\delta}}^{-\eta},\omega}$, which is well-defined and bounded from (\ref{apriori}). 
The Hahn--Banach theorem allows us to extend $\phi$ to a bounded linear functional on $L^2_{n,n-1}(\Omega, {\hat{\delta}}^{-\eta}, \omega)$ 
and the Riesz representation theorem yields $u \in L^2_{n,n-1}(\Omega,{\hat{\delta}}^{-\eta},\omega)$ satisfying 
$$\llangle \opa^\ast v, u\rrangle_{{\hat{\delta}}^{-\eta},\omega} = \llangle v, f \rrangle_{{\hat{\delta}}^{-\eta},\omega}$$
for all $v \in \mathcal{D}^{n,n}(\Omega)$, i.e. $\opa u = f$ in the distribution sense in $\Omega$. \\

%%%
Let us proceed to prove (\ref{apriori}).  For a direct proof of it, we would have to work with different adjoint operators. 
Therefore it is somewhat more convenient to actually prove the dual a priori estimate
\begin{equation}  \label{dualapriori}
\Vert v\Vert_{{\hat{\delta}}^{\eta},\omega} \leq C \Vert \opa v\Vert_{{\hat{\delta}}^{\eta},\omega}
\end{equation}
for any $v\in \mathcal{D}^{0,0}(\Omega)$. (\ref{apriori}) then follows from (\ref{dualapriori}) using a weighted Hodge star operator.\\

%%%
So let us proceed to prove (\ref{dualapriori}). 
Since $\eta <\eta_{\hat{\delta}}$, there exists some small $\varepsilon > 0$ such that $\eta + \varepsilon < \eta_{\hat{\delta}}$, which means that
$$i\pa\opa (-{\hat{\delta}}^{\eta +\varepsilon}) \geq 0\quad\mathrm{in}\ \Omega.$$
But then
$$i\pa\opa\log{\hat{\delta}}^{\eta +\varepsilon} = \frac{i\pa\opa{\hat{\delta}}^{\eta +\varepsilon}}{{\hat{\delta}}^{\eta +\varepsilon}} -i\pa\log{\hat{\delta}}^{\eta +\varepsilon}\wedge\opa\log{\hat{\delta}}^{\eta +\varepsilon}\leq -i\pa\log{\hat{\delta}}^{\eta +\varepsilon}\wedge\opa\log{\hat{\delta}}^{\eta +\varepsilon}.$$
Hence we get 
\[
  \mathrm{Trace}_\omega (i\pa\opa \log{\hat{\delta}}^{\eta +\varepsilon}) \leq -\vert\opa\log{\hat{\delta}}^{\eta +\varepsilon}\vert^2_\omega \text{\qquad in\ $\Omega$}
\]
Putting $\psi = {\hat{\delta}}^\eta$, we have $i\pa\opa\psi = -\omega$ by definition of $\omega$, thus $\mathrm{Trace}_\omega(i\pa\opa\psi) = -n$. Hence we get
\begin{equation}  \label{trace}
\mathrm{Trace}_\omega (i\pa\opa\psi + i\pa\opa\log{\hat{\delta}}^{\eta + \varepsilon}) \leq -n-\vert \pa\log{\hat{\delta}}^{\eta + \varepsilon}\vert^2_\omega \qquad\text{\ on\ $\Omega$}.
\end{equation}
 \\

%%%
On $\Omega$, we consider the weight function $e^{-\psi}$. Since $e^{-\psi}$ is bounded from below and from above by positive constants on $\Omega$, 
we can replace the norm $\Vert\cdot\Vert$ by $\Vert\cdot\Vert_{e^{-\psi}}$ for forms on $\Omega$. 

Multiplying the metric of the trivial bundle $\C$  further by ${\hat{\delta}}^{-(\eta +\varepsilon)} = e^{-\log{\hat{\delta}}^{\eta +\varepsilon}}$ on $\Omega$, 
it then follows from (\ref{BKN}) and (\ref{negativecurvature}) that for  $u\in \mathcal{D}^{0,0}(\Omega)$ one has
$$\llangle -\mathrm{Trace}_\omega(i\pa\opa\psi + i\pa\opa\log {\hat{\delta}}^{\eta +\varepsilon})u,u\rrangle_{e^{-\psi}{\hat{\delta}}^{-(\eta+\varepsilon)},\omega} \ \leq \ \Vert \opa u\Vert^2_{e^{-\psi}{\hat{\delta}}^{-(\eta + \varepsilon)},\omega}.$$
Using (\ref{trace}) we obtain 
$$\llangle (n+\vert\opa\log{\hat{\delta}}^{\eta +\varepsilon}\vert^2_\omega)u,u\rrangle_{e^{-\psi}{\hat{\delta}}^{-(\eta + \varepsilon)},\omega}\  \leq 
\ \Vert \opa u\Vert^2_{e^{-\psi}{\hat{\delta}}^{-(\eta +\varepsilon)},\omega}$$
for  $u\in \mathcal{D}^{0,0}(\Omega)$.
Observing that $\pa\log{\hat{\delta}}^{\eta+\varepsilon} = (\eta+\varepsilon)\pa\log{\hat{\delta}}$ and
setting $u= v{\hat{\delta}}^{\eta +\varepsilon/2}$ we obtain
 \begin{eqnarray}
 \llangle (n+ (\eta+\varepsilon)^2 \vert\opa\log{\hat{\delta}}\vert^2_\omega)v,v\rrangle_{e^{-\psi}{\hat{\delta}}^{\eta},\omega}  \leq  \Vert \opa v + (\eta+\frac{\varepsilon}{2})v\opa\log{\hat{\delta}}\Vert^2_{e^{-\psi}{\hat{\delta}}^{\eta},\omega}\nonumber\\
  \leq  (1+\frac{1}{a}) \Vert \opa v\Vert^2_{e^{-\psi}{\hat{\delta}}^{\eta},\omega} + (1+a)(\eta +\frac{\varepsilon}{2})^2 \Vert v\opa\log {\hat{\delta}}\Vert^2_{e^{-\psi}{\hat{\delta}}^{\eta},\omega}.
  \label{estimate}
 \end{eqnarray}
Choosing $a$ so small that $(1+a) (\eta +\frac{\varepsilon}{2})^2 \leq (\eta +\varepsilon)^2$, we can thus absorb the last term in (\ref{estimate}) in the left hand side, which immediately gives the a priori estimate (\ref{dualapriori}).\\
\end{proof}

%%%%%
Now let us give the proof of Theorem \ref{vanishing}.

\begin{proof}[Proof of Theorem \ref{vanishing}] 
By the assumption on $\Omega$, we can find a defining function $-{\hat{\delta}}$ with $\eta_{\hat{\delta}} > \frac{n-\ell}{n}$. We fix some real $\eta$ such that $\frac{n-\ell}{n} < \eta < \eta_{\hat{\delta}}$ and apply Theorem \ref{donnelly-fefferman} with this choice of $\eta$.

Now let $f\in L^2_{n,n}(X)$ be compactly supported in $\Omega$, which implies that $f\in L^2_{n,n}(\Omega,{\hat{\delta}}^{-\eta},\omega)$. Hence it follows from Theorem \ref{donnelly-fefferman} that there exists $u\in L^2_{n,n-1}(\Omega,{\hat{\delta}}^{-\eta},\omega)$ satisfying $\opa u =f$ in $\Omega$. \\

%%%
We first claim that if we extend $u$ by zero outside $\Omega$, then it defines a current $T = T_u \in\mathcal{D}^\prime_{0,1}(X)$. 
Indeed, we see from Lemma \ref{formdev} that 
$$\int_\Omega \vert u\vert^2_{\omega_{0}} {\hat{\delta}}^{1-\nu} dV_{\omega_{0}} 
 \lesssim  \int_\Omega \vert u\vert^2_\omega {\hat{\delta}}^{1-\nu} {\hat{\delta}}^{(n-1)\eta -2 -(n-\ell -1)} dV_\omega .$$
 Now a straightforward computation shows that the last integral can be estimated by $\int_\Omega \vert u\vert^2_\omega{\hat{\delta}}^{-\eta} dV_\omega < +\infty$ if $\nu \leq n\eta -n + \ell$. But by assumption on $\eta$ we have $n\eta - n + \ell > 0$, hence we may deduce that for some small $\nu > 0$ we have 
 $u\in L^2_{n,n-1}(\Omega,{\hat{\delta}}^{1-\nu})$.\\
 
  But then for any $v\in\mathcal{C}^\infty_{0,1}(X)$ we have
\begin{eqnarray}
\vert\int_\Omega u\wedge v\vert^2 & \leq & (\int_\Omega \vert u\vert^2_{\omega_{0}} {\hat{\delta}}^{1-\nu} dV_{\omega_{0}}) \cdot (\int_\Omega \vert v\vert^2_{\omega_{0}} {\hat{\delta}}^{-1+\nu} dV_{\omega_{0}})\nonumber\\
& \leq & \Vert u\Vert^2_{{\hat{\delta}}^{1-\nu}} \cdot (\int_{\Omega} {\hat{\delta}}^{-1+\nu} dV_{\omega_{0}}) \sup_\Omega \vert v\vert^2_{\omega_{0}}.\label{current}
\end{eqnarray}
Since $\nu > 0$, we have $\int_{\Omega} {\hat{\delta}}^{-1+\nu} dV_{\omega_{0}}< + \infty$. Therefore, $u$ defines a current $T\in\mathcal{D}^\prime_{0,1}(X)$.\\

%%%
It remains to see that $T= T_u$ satisfies $\opa T = f$ in the sense of distributions\texttt{} on $X$. 
Let $\alpha\in \mathcal{C}^\infty_{0,0}(X)$. We must show that
\begin{equation}  \label{toshow}
\int_\Omega u\wedge \opa\alpha = \int_\Omega f\wedge\alpha.
\end{equation}

Let $\chi\in\mathcal{C}^\infty(\R,\R)$ be a function such that $\chi(t)=0$ for $t \leq \frac{1}{2}$ and $\chi(t)=1$ for $t \geq 1$. Set $\chi_j = \chi(j{\hat{\delta}})\in\mathcal{D}^{0,0}(\Omega)$. 
Then $\chi_j\alpha\in\mathcal{D}^{0,0}(\Omega)$, and since $\opa u = f$ in $\Omega$, we therefore have
$$\int_\Omega f\wedge \chi_j\alpha = \int_\Omega u\wedge\opa (\chi_j\alpha) = \int_\Omega u\wedge(\alpha\opa \chi_j + \chi_j\wedge\opa\alpha)$$
As $f$ has $L^2$ coefficients on $\Omega$, the integral of $f\wedge\chi_j\alpha$   converges  to the integral of $f\wedge\alpha$  as $j$ tends to infinity.  The convergence of  the integral of $u\wedge\chi_j\opa\alpha$ to the integral of $u\wedge\opa\alpha$ follows from $u\in L^2_{n,n-1}(\Omega,{\hat{\delta}}^{1-\nu})$ (use the estimate (\ref{current})). 

The remaining term can be estimated as follows: Using the Cauchy-Schwarz inequality we have
\begin{align*}
\vert\int_\Omega u\wedge\alpha\opa\chi_j\vert^2 
%&= \left\vert\int_{\lbrace \frac{1}{2j} \leq {\hat{\delta}} \leq \frac{1}{j}\rbrace} u {\hat{\delta}}^{-\eta/2} \wedge\alpha\opa\chi_j {\hat{\delta}}^{\eta/2} \right\vert^2 \\
%&= \left\vert \int_{\lbrace \frac{1}{2j} \leq {\hat{\delta}} \leq \frac{1}{j}\rbrace} u {\hat{\delta}}^{-\eta/2} \wedge \star_\omega \star_\omega \alpha \opa\chi_j {\hat{\delta}}^{\eta/2} \right\vert^2 \\
&= \left\vert \int_{\lbrace \frac{1}{2j} \leq  {\hat{\delta}} \leq \frac{1}{j}\rbrace} \langle u {\hat{\delta}}^{-\eta/2}, \overline{\star_\omega \alpha \opa\chi_j {\hat{\delta}}^{\eta/2}} \rangle_\omega dV_\omega \right\vert^2 \\
&\leq \int_{\lbrace \frac{1}{2j} \leq {\hat{\delta}} \leq \frac{1}{j}\rbrace} |u {\hat{\delta}}^{-\eta/2}|_\omega^2 dV_\omega 
      \cdot \int_{\lbrace \frac{1}{2j} \leq {\hat{\delta}} \leq \frac{1}{j}\rbrace} |\star_\omega \alpha \opa\chi_j {\hat{\delta}}^{\eta/2}|_\omega^2 dV_\omega  \\
%&= \int_{\lbrace \frac{1}{2j} \leq {\hat{\delta}} \leq \frac{1}{j}\rbrace} |u|_\omega^2 {\hat{\delta}}^{-\eta} dV_\omega 
%      \cdot \int_{\lbrace \frac{1}{2j} \leq {\hat{\delta}} \leq \frac{1}{j}\rbrace} |\alpha \opa\chi_j|_\omega^2 {\hat{\delta}}^{\eta} dV_\omega  \\
&\leq \sup_\Omega\vert\alpha\vert^2 \int_{\lbrace {\hat{\delta}} \leq \frac{1}{j}\rbrace} \vert u\vert^2_{\omega}{\hat{\delta}}^{-\eta}dV_{\omega} \cdot 
      \int_\Omega\vert\opa\chi_j\vert^2_{\omega}{\hat{\delta}}^{\eta} dV_{\omega}
\end{align*}
%$$\vert\int_\Omega u\wedge\alpha\opa\chi_j\vert^2 
%%\leq \sup_\Omega\vert\alpha\vert^2 \int_{\lbrace{\hat{\delta}} \leq \frac{1}{j}\rbrace} \vert u\vert^2_{\omega}{\hat{\delta}}^{-\eta}dV_{\omega} \cdot 
%%     \int_\Omega\vert\opa\chi_j\vert^2_{\omega}{\hat{\delta}}^{\eta} dV_{\omega}.$$
where $\star_\omega$ denotes the Hodge star operator with respect to $\omega$ in $\Omega$. 
Since $u\in L^2_{n,n-1}(\Omega,{\hat{\delta}}^{-\eta},\omega)$, 
the integral $\int_{\lbrace{\hat{\delta}} \leq \frac{1}{j}\rbrace} \vert u\vert^2_{\omega}{\hat{\delta}}^{-\eta}dV_{\omega}$ converges to $0$ when $j$ tends to infinity. 

To estimate the second integral, we look at the behavior of its integrand $|\opa \chi_j|^2_\omega$ near $\pa\Omega$. 
From $\opa\chi_j = j\chi^\prime\opa{\hat{\delta}}$, 
\begin{align*}
{\hat{\delta}}(z)^{\eta-2}|\opa\chi_j|^2_\omega(z)
& \leq j^2 \| \chi' \|^2_{\mathcal{C}^1(\R)} |\opa{\hat{\delta}}|^2_{{\hat{\delta}}^{2-\eta} \omega}(z) \\
& = j^2 \| \chi' \|^2_{\mathcal{C}^1(\R)} \max_{0 \neq v \in T^{1,0}_z X} \frac{|\pa{\hat{\delta}}(v)|^2}{\eta ({\hat{\delta}}(z) i\pa\opa(-{\hat{\delta}})(v,v) + |\pa{\hat{\delta}}(v)|^2)} \\
& \to j^2 \| \chi'\|^2_{\mathcal{C}^1(\R)} \frac{1}{\eta} \text{\quad as $z \to \pa\Omega$.}
\end{align*}
Therefore, $|\opa\chi_j|^2_\omega \lesssim j^2 {\hat{\delta}}^{2-\eta}$ near $\pa\Omega$. 
Since the Levi form of $\pa\Omega$ has $\ell$ zero eigenvalues, we can estimate it with Lemma \ref{volumedev} as: 
\begin{eqnarray*}
\int_\Omega \vert \opa\chi_j\vert^2_\omega {\hat{\delta}}^{\eta}dV_\omega 
& \lesssim & \int_{\lbrace{\hat{\delta}} \leq \frac{1}{j}\rbrace} j^2 {\hat{\delta}}^{2-\eta} {\hat{\delta}}^{\eta} {\hat{\delta}}^{n\eta -2 - (n-\ell-1)} dV_{\omega_{0}} \\
& =   & \int_{\lbrace{\hat{\delta}} \leq \frac{1}{j}\rbrace} j^2 {\hat{\delta}}^{1 + n\eta - (n-\ell)} dV_{\omega_{0}} \\
& \lesssim & j^{2 - (2 + n\eta - (n-\ell))}     \\
& =        & j^{-n\eta + n-\ell} \to 0
\end{eqnarray*}
as $j \to \infty$ since $-n\eta + n-\ell < 0$ by the assumption $\eta > \frac{n-\ell}{n}$.

Therefore, $\int_\Omega u\wedge\alpha\opa\chi_j$ converges to $0$ when $j$ tends to infinity. Equation (\ref{toshow}) follows. 
\end{proof}

%%%%%%%%%%%%%%%%%%%%%%%%%%%%%%%%%%%%%%%%%%%%%%%%%%%%%%%%%%%%%%%%%%%%%%%%%%%%%%%%%%%
\section{Proof of the main theorem}

The proof of Main Theorem easily follows from Theorem \ref{vanishing} using a duality argument.

\begin{proof}[Proof of Main Theorem]
Assume by contradiction  that the Levi-form of the boundary $\pa \Omega$ has $\ell$ zero eigenvalues, and assume that $\eta(\Omega) > \frac{n-\ell}{n}$.
Let $f\in\mathcal{D}^{n,n}(\Omega)$ be a smooth form of top degree with compact support in $\Omega$ satisfying $\int_\Omega f=1$. 
Applying Theorem \ref{vanishing}, we can find a current $T\in\mathcal{D}^\prime_{0,1}(X)$ satisfying $\opa T = f$ in the current sense. 
Let $\chi$ be a compactly supported smooth function on $X$ which is equal to one on $\ol{\Omega}$. But then
$$1= \int_\Omega f = \langle f,{\chi} \rangle = \langle T,\opa \chi \rangle = 0.$$
This contradiction proves that $\eta(\Omega) \leq \frac{n-\ell}{n}$.
\end{proof}


\begin{thebibliography}
\footnotesize

\bibitem[A]{A} \textsc{M. Adachi:} \emph{On the ampleness of positive CR line bundles over Levi-flat manifolds.} Publ. Res. Inst. Math. Sci. {\bf 50}, 153--167 (2014).

\bibitem[BCh]{BCh}  \textsc{B. Berndtsson, Ph. Charpentier:} \emph{A Sobolev mapping property for the Bergman kernel.} Math. Z. {\bf 235}, 1--10 (2000).

\bibitem[Br]{Br} \textsc{J. Brinkschulte:} \emph{The {$\overline\partial$}-problem with support conditions on some weakly pseudoconvex domains}, Ark. Mat. {\bf 42}, 259--282 (2004).

\bibitem[CSh]{CSh} \textsc{D. Chakrabarti, M.-C. Shaw:} \emph{$L^2$ {S}erre duality on domains in complex manifolds and applications.} Trans. Amer. Math. Soc. {\bf 364}, 3529--3554 (2012).

\bibitem[CShW]{CShW} \textsc{J.Cao, M.-C. Shaw, L. Wang:} \emph{Estimates for the $\opa$-Neumann problem and nonexistence of $\mathcal{C}^2$ Levi-flat hypersurfaces in $\C\PP^n$.} Math. Z. {\bf 248}, 183--221, Erratum, 223--225 (2004).

\bibitem[De]{D} \textsc{J.-P. Demailly:} \emph {Complex analytic and
differential geometry.} Available at
\emph{http://www-fourier.ujf-grenoble.fr/$\sim$demailly/books.html}


\bibitem[DiFo1]{DiFo1} \textsc{K. Diederich, J. E. Fornaess:} \emph{Pseudoconvex domains: bounded strictly plurisubharmonic exhaustion functions.} Invent. Math. {\bf 39}, 129--141 (1977).


\bibitem[DiFo2]{DiFo2} \textsc{K. Diederich, J. E. Fornaess:} \emph{Pseudoconvex domains: an example with nontrivial Nebenh\"ulle.} Math. Ann. {\bf 225}, 275--292 (1977).


\bibitem[DoFe]{DoFe} \textsc{H. Donnelly, C. Fefferman:} \emph{$L^2$ cohomology and index theorem for the Bergman metric.} Ann. Math. {\bf 118}, 593--618 (1983).

\bibitem[FuSh]{FuSh} \textsc{S. Fu, M.-C. Shaw:} \emph{The Diederich-Forn{\ae}ss exponent and non-existence of Stein domains with Levi-flat boundaries}, J. Geom. Anal., published online on 25 November 2014.

\bibitem[HSh]{HSh} \textsc{P. S. Harrington, M.-C. Shaw:} \emph{The strong Oka's lemma, bounded plurisubharmonic functions and the $\overline{\partial}$-Neumann problem}, Asian J. Math. {\bf 11}, 127--139 (2007).


\bibitem[OSi]{OSi}  \textsc{T. Ohsawa, N. Sibony:} \emph{Bounded P.S.H. functions and pseudoconvexity in K\"ahler manifolds.} Nagoya Math. J. {\bf 149}, 1--8 (1998).



\end{thebibliography}
\end{document}